\documentclass[12pt]{amsart}
\usepackage{fullpage}
\usepackage{amsmath,amsbsy,amsrefs,amssymb,hyperref}
\usepackage{amsmath}
\usepackage{amsthm}
\usepackage{amsbsy}
\usepackage{amssymb}
\usepackage{hyperref}
\usepackage{amsfonts}
\usepackage{color}

\usepackage{enumitem}
\numberwithin{equation}{section}
\newtheorem{theorem}{Theorem}[section]

\newtheorem{proposition}[theorem]{Proposition}
\newtheorem{corollary}[theorem]{Corollary}

\newtheorem*{conjecture*}{Conjecture}
\newtheorem*{question*}{Serre's Question}
\newtheorem{hypothesis}[theorem]{Hypothesis}
\theoremstyle{remark}
\newtheorem*{remark*}{Remark}

\makeatletter
\renewcommand{\pod}[1]{\mathchoice
  {\allowbreak \if@display \mkern 18mu\else \mkern 8mu\fi (#1)}
  {\allowbreak \if@display \mkern 18mu\else \mkern 8mu\fi (#1)}
  {\mkern4mu(#1)}
  {\mkern4mu(#1)}
}

\makeatletter
\newcommand{\Z}{\mathbb{Z}}

\newcommand{\Q}{\mathbb{Q}}
\newcommand{\F}{\mathbb{F}}
\newcommand{\N}{\mathbb{N}}

\newcommand{\p}{\mathfrak{p}}
\renewcommand{\a}{\mathfrak{a}}

\newcommand{\rk}{\textup{rk}}
\newcommand{\f}{\mathfrak{f}}

\begin{document}
  \title[A Short-Interval Bombieri-Vinogradov Theorem]
    {A Variant of the Bombieri-Vinogradov Theorem in Short Intervals With Applications}
\author{Jesse Thorner}

  \maketitle

  \begin{abstract}
We generalize the classical Bombieri-Vinogradov theorem to a short interval, non-abelian setting.  This leads to variants of the prime number theorem for short intervals where the primes lie in arithmetic progressions that are ``twisted'' by a splitting condition in a Galois extension $L/K$ of number fields.  Using this result in conjunction with recent work of Maynard, we prove that rational primes in short intervals with a given splitting condition in a Galois extension $L/\Q$ exhibit dense clusters in short intervals.  We explore several arithmetic applications related to questions of Serre regarding the nonvanishing Fourier coefficients of cuspidal modular forms, including finding dense clusters of fundamental discriminants $ d $ in short intervals for which the central values of $d$-quadratic twists of modular $L$-functions are non-vanishing.
  \end{abstract}

%%%%%%%%%%%%%%%%%%%%%%%%%%%%%%%%%%%%%%%%
\section{Introduction and Statement of Results}
\label{sec:intro}
%%%%%%%%%%%%%%%%%%%%%%%%%%%%%%%%%%%%%%%%

Let $\N$ denote the set of positive integers, and let $a,q\in\N$ satisfy $(a,q)=1$.  Define
\[
\psi(x;q,a)=\sum_{\substack{n\leq x \\ n\equiv a\pmod q}}\Lambda(n),
\]
where $\Lambda(n)$ is the von Mangoldt function.  The prime number theorem for arithmetic progressions tells us that if $q\leq (\log x)^D$ for any constant $D>0$, we have
\begin{equation}
\label{pnt-prog}
\psi(2x;q,a)-\psi(x;q,a)\sim\frac{x}{\varphi(q)},
\end{equation}
where $\varphi$ denotes Euler's totient function.  Understanding both the error term and the range of $q$ for (\ref{pnt-prog}) is important for a wide variety of arithmetic problems.  The Generalized Riemann Hypothesis (GRH) for Dirichlet $L$-functions implies that if $q\leq x^{1/2-o(1)}$, then
\begin{equation}
\label{1.2}
\psi(2x;q,a)-\psi(x;q,a)-\frac{x}{\varphi(q)}\ll\sqrt{x}(\log qx)^2.
\end{equation}
While this is beyond the reach of current methods, it is known that the mean value of (\ref{1.2}) is about as small as predicted by GRH when we average over moduli $q$.  More specifically, if $0\leq\theta<\frac{1}{2}$ is constant, Bombieri and Vinogradov proved that for any fixed $D>0$, we have
\begin{equation}
\label{BV}
\sum_{q\leq x^{\theta}}\max_{ (a,q)=1}\max_{N\leq x}\left|\psi(2N;q,a)-\psi(N;q,a)-\frac{N}{\varphi(q)}\right|\ll \frac{x}{(\log x)^D}.
\end{equation}

A more difficult problem asks for the distribution of primes in arithmetic progressions when the interval $[x,2x]$ is replaced with $[x,x+h]$, where $h\geq x^{1-\delta}$ for some $\delta>0$.  Using deep analytic properties of Dirichlet $L$-functions, one can produce a short interval analogue of the Bombieri-Vinogradov estimate \eqref{BV} of the form
\begin{equation}
\label{BV-short}
\sum_{q\leq x^\theta}\max_{(a,q)=1}\max_{y\leq h}\max_{\frac{1}{2}x\leq N\leq x}\left|\psi(N+y;q,a)-\psi(N;q,a)-\frac{y}{\varphi(q)}\right|\ll \frac{h}{(\log x)^D},
\end{equation}
where $D>0$, $\delta>0$, and $\theta>0$ are constants and $h\geq x^{1-\delta}$.  The Density Hypothesis for Dirichlet $L$-functions, which follows from GRH, predicts that \eqref{BV-short} holds when $0\leq\delta<\frac{1}{2}$ and $0\leq\theta<\frac{1}{2}-\delta$ \cite[Chapter 12]{Montgomery-topics}.  There has been much progress toward this conjectured estimate; see \cite{PPS1} and the sources contained therein.  Currently, the sharpest version of \eqref{BV-short} is due to Timofeev \cite{Timofeev}, who proved that \eqref{BV-short} holds when $0\leq\delta<\frac{5}{12}$ and
\[
0\leq\theta<\begin{cases}
\frac{1}{2}-\delta&\mbox{if $0\leq\delta<\frac{2}{5}$},\\
\frac{9}{20}-\delta&\mbox{if $\frac{2}{5}\leq\delta<\frac{5}{12}$}.
\end{cases}
\]

Some of these results have been extended to a Chebotarev setting.  Specifically, let $L/K$ be a Galois extension of number fields with Galois group $G$ and absolute discriminant $d_L$, let $a,q\in\N$ with $ (a,q)=1$, and let $\mathrm{N}=\mathrm{N}_{K/\Q}$ denote the absolute field norm of $K$.  For a prime ideal $\p$ of $K$ which is unramified in $L$, there corresponds a certain conjugacy class $C\subset G$ consisting of the set of Frobenius automorphisms attached to the prime ideals of $L$ which lie over $\p$.  We denote this conjugacy class by the Artin symbol $[\frac{L/K}{\p}]$.  

For a fixed conjugacy class $C$ and an integral ideal $\a$ of $K$, define
\[
\Lambda_C(\a)=\begin{cases}
\log\mathrm{N}\p&\mbox{if $\a=\p^m$ with $m\geq1$, $\p$ unramified in $L$, and {\small$\left[\frac{L/K}{\p}\right]^m=C$}},\\
0&\mbox{otherwise}
\end{cases}
\]
and
\[
\psi_{C}(x;q,a)=\sum_{\substack{\mathrm{N}\a\leq x \\ \mathrm{N}\a\equiv a\pmod q}}\Lambda_C(\a).
\]
The Chebotarev density theorem tells us that if $q\leq(\log x)^D$, then
\begin{equation}
\label{eqn:cdt}
\psi_{C}(2x;q,a)-\psi_{C}(x;q,a)\sim d(C;q,a)x
\end{equation}
for some density $d(C;q,a)\geq0$.  If $(q,d_L)=1$, then
\[
d(C;q,a)=\frac{|C|}{|G|}\frac{1}{\varphi(q)}.
\]
In the case of $q=1$, Balog and Ono \cite{BO} extended \eqref{eqn:cdt} to a short interval setting by proving that if we fix
\begin{equation}
\label{BO-short1}
0<\delta<\begin{cases}
1/[L:\Q]&\mbox{if $[L:\Q]\geq3$,}\\
3/8&\mbox{if $[L:\Q]=2$,}\\
5/12&\mbox{if $[L:\Q]=1$}
\end{cases}
\end{equation}
and choose $h\geq x^{1-\delta}$, then
\begin{equation}
\label{BO-short}
\psi_{C}(x+h;1,1)-\psi_{C}(x;1,1)\sim\frac{|C|}{|G|}h.
\end{equation}

Building on the work of M. Ram Murty and V. K. Murty \cite{MM-chebotarev}, M. Ram Murty and Petersen \cite{MP} proved that if  $H\subset G$ is the largest abelian subgroup of $G$ such that $H\cap C$ is nonempty, $E$ is the fixed field of $H$, and $0\leq\theta<1/\max\{[E:\Q]-2,2\}$ is fixed, then
\begin{equation}
\label{murty-bv}
\sideset{}{'}\sum_{q\leq x^\theta}\max_{ (a,q)=1}\max_{N\leq x}\left|\psi_{C}(2N;q,a)-\psi_{C}(N;q,a)-\frac{|C|}{|G|}\frac{N}{\varphi(q)}\right|\ll \frac{x}{(\log x)^D},
\end{equation}
where $\sum'$ denotes summing over moduli $q$ satisfying $(q,d_L)=1$.  This extends \eqref{BV} to a Chebotarev setting; in fact, \eqref{BV} is recovered when $L=\Q$.  Our main result is a Chebotarev analogue of \eqref{BV-short}, which we prove in Section \ref{sec:proof_main_thm}.

\begin{theorem}
\label{main-theorem}
Let $L/K$ be a Galois extension of number fields with Galois group $G$ and absolute discriminant $d_L$, and let $C\subset G$ be a fixed conjugacy class.  Let $H\subset G$ be the largest abelian subgroup of $G$ such that $H\cap C$ is nonempty, and let $E$ be the fixed field of $H$.  Fix $0\leq\delta<\frac{2}{5[E:\Q]}$ and $0\leq\theta<\frac{1}{3}(\frac{2}{5[E:\Q]}-\delta)$.  If $h\geq x^{1-\delta}$, then for any fixed $D>0$, we have
\[
\sideset{}{'}\sum_{q\leq x^{\theta}}\max_{ (a,q)=1}\max_{y\leq h}\max_{\frac{1}{2}x\leq N\leq x}\left|\psi_{C}(N+y;q,a)-\psi_{C}(N;q,a)-\frac{|C|}{|G|}\frac{y}{\varphi(q)}\right|\ll\frac{h}{(\log x)^D},
\]
where $\sum'$ denotes summing over moduli $q$ satisfying $(q,d_L)=1$.
\end{theorem}

The following improvement on the range of $\delta$ \eqref{BO-short1} in Balog and Ono's short interval version of the Chebotarev density theorem \eqref{BO-short} follows immediately from Theorem \ref{main-theorem}.

\begin{corollary}
\label{thm:CDT_short}
Let $L/K$ be a Galois extension of number fields with Galois group $G$ and absolute discriminant $d_L$, and let $C\subset G$ be a fixed conjugacy class.  Suppose that $[L:\Q]\geq3$.  Let $H\subset G$ be the largest abelian subgroup of $G$ such that $H\cap C$ is nonempty, and let $E$ be the fixed field of $H$.  Suppose that $q\leq (\log x)^D$ satisfies $(q,d_L)=1$ and $(a,q)=1$.  If $0\leq\delta<\max\{\frac{1}{[L:\Q]},\frac{2}{5[E:\Q]}\}$ is fixed and $h\geq x^{1-\delta}$, then
\[
\psi_C(x+h;q,a)-\psi_C(x;q,a)\sim\frac{|C|}{|G|}\frac{h}{\varphi(q)}.
\]
\end{corollary}

Much like the results of \cite{MM-chebotarev,MP}, nonabelian analogues of the Bombieri-Vinogradov theorem in short intervals can have interesting arithmetic consequences.  In this paper, we will focus on consequences related to recent advances toward the Hardy-Littlewood prime $k$-tuples conjecture.  For these applications, we consider a Galois extension $L/\Q$ with Galois group $G$ and absolute discriminant $d_L$, and we consider a fixed conjugacy class $C\subset G$.  In this setting, a Chebotarev set takes the form
\begin{equation}
\label{cheb-set}
\mathcal{P}=\left\{p:p\nmid d_L,\left[\frac{L/\Q}{p}\right]=C\right\}
\end{equation}

We establish some additional notation.  Let $\mathbb{P}$ denote the set of all primes, and let $h_i$ denote a nonnegative integer.  We call a collection of linear forms $\mathcal{H}_k=\{n+h_1,\ldots,n+h_k\}$ {\it admissible} $\prod_{i=1}^k (n+h_i)$ has no fixed prime divisor.  (We could consider more general admissible sets $\{a_1 n+b_1,\ldots,a_k n+b_k\}$, but this sometimes hinders the applications we consider.)

\begin{conjecture*}[Hardy-Littlewood]
If $\mathcal{H}_k$ is admissible, then as $x\to\infty$, we have
\[
\#\{n\in[x,2x]:\#(\{n+h_1,\ldots,n+h_k\}\cap\mathbb{P})=k\}\sim \mathfrak{S}\frac{x}{(\log x)^k},
\]
where $\mathfrak{S}$ is a certain positive constant depending on $\mathcal{H}_k$.

\end{conjecture*}
\noindent
Choosing $\mathcal{H}_2=\{0,2\}$, the Hardy-Littlewood conjecture implies the elusive twin prime conjecture, that there are infinitely many pairs of primes whose difference is 2.

In \cite{maynard}, Maynard developed a significant improvement to the Selberg sieve.  By using this improvement in conjunction with \eqref{BV}, Maynard proved that if $\mathcal{H}_{k}$ is admissible, then there are infinitely many integers $N>0$ such that for some $n\in[N,2N]$, we have
\[
\#(\{n+h_1,\ldots,n+h_k\}\cap\mathbb{P})\geq (1/4+o_{k\to\infty}(1))\log k.
\]
(Tao independently derived the same improvement as Maynard at roughly the same time, but arrived at slightly different conclusions.)  Using \eqref{murty-bv} and Maynard's improvement to the Selberg sieve, the author \cite{JT} proved that if $\mathcal{H}_{k}$ is admissible, then there are infinitely many integers $N>0$ such that for some $n\in[N,2N]$, we have
\[
\#(\{n+h_1,\ldots,n+h_k\}\cap\mathcal{P})\geq \left(\frac{1}{2}\min\left\{\frac{1}{2},\frac{2}{|G|}\right\}\frac{|C|}{|G|}\frac{\varphi(d_L)}{d_L}+o_{k\to\infty}(1)\right)\log k,
\]
where $\mathcal{P}$ is a Chebotarev set given by \eqref{cheb-set}.  The author explored applications of this result to ranks of quadratic twists of elliptic curves, congruence conditions on the Fourier coefficients of newforms, and representations of primes by binary quadratic forms.

In \cite{maynard2}, Maynard generalized his methods to prove weak forms of the Hardy-Littlewood conjecture with specializations to primes in short intervals and primes in Chebotarev sets.  More specifically, given $0\leq\delta<\frac{5}{12}$ and $h\geq x^{1-\delta}$, Maynard proved that there exists an absolute constant $C>0$ such that if $k\geq C$ and $\mathcal{H}_k$ is an admissible set, then
\begin{equation}
\label{thing-1}
\#\{n\in[x,x+h]:\#(\{n+h_1,\ldots,n+h_k\}\cap\mathbb{P})\geq C^{-1}\log k\}\gg\frac{h}{(\log x)^k}
\end{equation}
Furthermore, if $\mathcal{P}$ is given by \eqref{cheb-set}, then Maynard also proved that there exists a constant $C_L>0$ such that if $k\geq C_L$ and $\mathcal{H}_{k}$ is admissible, then
\begin{equation}
\label{thing-2}
\#\{n\in[x,2x]:\#(\{n+h_1,\ldots,n+h_k\}\cap\mathcal{P})\geq C_L^{-1}\log k\}\gg\frac{x}{(\log x)^{k}}.
\end{equation}
(The subscript $L$ in $C_L$ denotes that the constant $C$ depends only on $L$ in an effectively computable fashion.  We will use this convention henceforth.)

Using Theorem \ref{main-theorem}, we prove in Section \ref{sec:proofs2} the following mutual refinement of \eqref{thing-1} and \eqref{thing-2}, which extends the author's applications in \cite{JT} to a short interval setting.

\begin{theorem}
\label{bounded-gaps-short}
Let $L/\Q$ be a Galois extension of number fields, let $\mathcal{P}$ be as in \eqref{cheb-set}, and choose $h$ as in Theorem \ref{main-theorem}.  There exists a constant $C_L\in\N$ such that if $k\geq C_L$ and $\mathcal{H}_{k}$ is admissible, then
\[
\#\{n\in[x,x+h]:\#(\{n+h_1,\ldots,n+h_k\}\cap\mathcal{P})\geq C_L^{-1}\log k\}\gg \frac{h}{(\log x)^{k}}.
\]
\end{theorem}

\begin{remark*}
Some of the parameters in the statement of Theorem \ref{bounded-gaps-short} can have some uniformity in $x$ by appealing to the arguments in \cite{maynard2}.  In what follows, all parameters are constant with respect to $x$.
\end{remark*}

We now consider arithmetic consequences of Theorem \ref{bounded-gaps-short} in the theory of elliptic curves, modular forms, and modular $L$-functions; for an introduction to the relevant definitions and ideas, we refer the reader to \cite{web}.  We consider the following question of Serre \cite{Ser1}, which may be seen as an automorphic analogue of Bertrand's postulate on the existence of primes in every dyadic interval $[x,2x]$.

\begin{question*}
Let $q=e^{2\pi i z}$, and let $S_{\ell}(\Gamma_0(N),\chi)$ denote the space of weight $\ell$, level $N$ cusp forms.  For a nonzero cusp form $f(z)=\sum_{n=1}^\infty a_f(n)q^n\in S_{\ell}(\Gamma_0(N),\chi),$ let
\[
I_f(n)=\max\{i:a_f(n+j)=0\textup{ for all }0\leq j\leq i\}.
\]
\begin{enumerate}
\item Suppose that $f$ is of weight $\ell\geq2$ and is not a linear combination of forms with complex multiplication.  Is $I_f(n)\ll n^\delta$ for some $0\leq\delta<1$?
\item More generally, are there analogous results for forms with non-integral weights, or forms with respect to other Fuchsian groups?
\end{enumerate}
\end{question*}

Motivated by the second part of Serre's question, Balog and Ono \cite{BO} used \eqref{BO-short} to prove that if $f(z)=\sum_{n=1}^\infty a_f(n)q^n\in S_{\ell}(\Gamma_0(N),\chi)$ is a cusp form of weight $\ell\in\frac{1}{2}\N-\{\frac{1}{2}\}$ which is not a linear combination of weight $\frac{3}{2}$ theta functions, then there exists $\nu_f\in\N$ such that if $0\leq\delta<\frac{1}{\nu_f}$ and $h\geq x^{1-\delta}$, then
\begin{equation}
\label{bo-coeff}
\#\{n\in[x,x+h]:a_f(n)\neq 0\}\gg \frac{h}{\log x}.
\end{equation}
For such a cusp form $f$, it follows that $I_f(n)\ll n^{1-\frac{1}{\nu_f}+\epsilon}$ for any $\epsilon>0$, affirmatively answering Serre's question.  By using Theorem \ref{bounded-gaps-short} instead of \eqref{BO-short} in Balog and Ono's proof, we immediately obtain dense clusters of integers $n$ in short intervals for which $a_f(n)\neq0$.  Specifically, we have the following.

\begin{theorem}
\label{short-fourier}
Let $f(z)=\sum_{n=1}^\infty a_f(n)q^n\in S_{\ell}(\Gamma_0(N),\chi)$ be a nonzero cusp form of weight $\ell\in\frac{1}{2}\N-\{\frac{1}{2}\}$ which is not a linear combination of weight $\frac{3}{2}$ theta functions.  There exist constants $C_f,\nu_f\in\N$ such that if $0\leq\delta<\frac{1}{\nu_f}$, $h\geq x^{1-\delta}$, $k\geq C_f$ and $\mathcal{H}_{k}$ is admissible, then
\[
\#\{n\in[x,x+h]:\#\{h_i\in\mathcal{H}_{k}:a_f(n+h_i)\neq0\}\geq C_f^{-1}\log k\}\gg \frac{h}{(\log x)^k}.
\]
\end{theorem}

We address two corollaries of Theorem \ref{short-fourier} regarding central values of modular $L$-functions and ranks of elliptic curves.  Let $\mathcal{D}$ be the set of all fundamental discriminants, and let $f(z)=\sum_{n=1}^\infty a_f(n)q^n\in S_{\ell}(\Gamma_0(N))$ be a newform (i.e., a holomorphic cuspidal normalized Hecke eigenform) of weight $\ell\in2\N$.  Given $ d \in\mathcal{D}$, let $L(s,f_{ d })$ denote the $L$-function given by
\[
L(s,f_{ d })=\sum_{n=1}^\infty \frac{a_f(n)\chi_{ d }(n)}{n^{s+(\ell-1)/2}},
\]
where $\chi_{ d }$ is the Kronecker character for $\Q(\sqrt{ d })$.  Goldfeld \cite{Goldfeld} conjectured that the density of $ d \in\mathcal{D}$ for which $L(1/2,f_{ d })\neq0$ is $1/2$.

By the work of Shimura \cite{Shimura} and Waldspurger \cite{Wald}, Fourier coefficients of half-integer weight cusp forms $g$ that satisfy the hypotheses of Theorem \ref{short-fourier} interpolate central values of quadratic twists of modular $L$-functions associated to the Shimura correspondent of $g$.  Despite the fact that the Shimura correspondence is not surjective, Ono and Skinner \cite{OS} proved that such central values can be obtained in this fashion for the $L$-function of an even-integer weight newform with trivial nebentypus.  Using this  observation along with \eqref{bo-coeff}, Balog and Ono \cite{BO} proved that there exists $\nu_F\in\N$ such that if $0\leq\delta<\frac{1}{\nu_F}$ and $h\geq x^{1-\delta}$, then
\begin{equation}
\label{ono-skinner}
\#\{|d|\in[x,x+h]: d \in\mathcal{D},~L(1/2,f_{d })\neq0\}\gg\frac{h}{\log x}.
\end{equation}
This is the sharpest result in the direction of Goldfeld's conjecture which is valid for all newforms $f$; slight improvements exist for certain classes of newforms \cite{Ono}.  By using Theorem \ref{short-fourier} instead of \eqref{bo-coeff} in Balog and Ono's proof, we immediately obtain dense clusters of fundamental discriminants $d$ in short intervals for which $L(1/2,f_d)\neq0$.

\begin{corollary}
\label{central-values}
Let $f\in S_{2\ell}(\Gamma_0(N))$ be a newform with $\ell\in\N$.  There exists an arithmetic progression $a\bmod q$ (which depends explicitly on $f$) and constants $\nu_f,C_f\in\N$ such that if $0\leq\delta<\frac{1}{\nu_f}$, $h\geq x^{1-\delta}$, $k\geq C_f$, $\mathcal{H}_{k}$ is admissible, and
\[
\mathcal{N}_f(k,n)=\{h_i\in\mathcal{H}_{k}:n+qh_i\in\mathcal{D},L(1/2,f_{n+qh_i})\neq0\},
\]
then
\[
\#\{|n|\in[x,x+h]:n\equiv a~(\bmod~q),\#\mathcal{N}_f(k,n)\geq C_f^{-1}\log k\}\gg\frac{h}{(\log x)^k}.
\]
\end{corollary}
\begin{remark*}
We need to restrict to the arithmetic progression $a\bmod q$ for technical reasons; see \cite{OS} for details.  We accomplish this by combining the arguments of Freiburg \cite[Proof of Theorem 1]{Freiburg} with Maynard's proofs in \cite{maynard2}, which is fairly straightforward.
\end{remark*}

Let $f$ be the newform associated to an elliptic curve $E/\Q$ of conductor $N$.  If $(d,4N)=1$, then $L(s,f_d)$ is the $L$-function of the $d$-quadratic twist $E_d/\Q$.  By the work of Kolyvagin \cite{Koly}, if $L(1/2,f_d)\neq0$, then the rank $\rk(E_d(\Q))$ of the Mordell-Weil group $E_d(\Q)$ is zero.  Thus Corollary \ref{central-values} immediately implies the existence of dense clusters of fundamental discriminants $d$ in short intervals such that $\rk(E_d)=0$.

\begin{corollary}
\label{twists}
Let $E/\Q$ be an elliptic curve.  There exist an arithmetic progression $a\bmod q$ (which depends explicitly on $E$) and constants $\nu_E,C_E\in\N$ such that if $0\leq\delta<\frac{1}{\nu_E}$, $h\geq x^{1-\delta}$, $k\geq C_E$, $\mathcal{H}_{k}$ is admissible, and
\[
\mathcal{N}_E(k,n)=\{h_i\in\mathcal{H}_{k}:n+qh_i\in\mathcal{D},\rk(E_{n+qh_i})=0\},
\]
then
\[
\#\left\{|n|\in[x,x+h]:n\equiv a~(\bmod~q),\#\mathcal{N}_E(k,n)\geq C_E^{-1}\log k\right\}\gg\frac{h}{(\log x)^k}.
\]
\end{corollary}

For our final application, consider an elliptic curve $E/\Q$.  In \cite{Murty2,Ser1}, the distribution of the quantity $a_E(p):=p+1-\#E(\F_p)$ is studied.  We apply our results to study the distribution of $a_E(p)\pmod m$ in short intervals, where $m$ is a given integer.  It follows from the work of Shiu \cite{Shiu} that if $E/\Q$ has a rational point of order $m$, then for every $j\in\N$ and every $i\not\equiv 1\pmod m$, there exists an $n\in\N$ such that
\[
a_E(p_n)\equiv a_E(p_{n+1})\equiv a_E(p_{n+2})\equiv\cdots\equiv a_E(p_{n+j})\equiv i\pmod m,
\]
where the primes are indexed in increasing order.  Using \eqref{BO-short} and the definition of the action of Galois on the torsion points of $E$, Balog and Ono \cite{BO} proved that if $m\in\N$ and $i\bmod m$ is a residue class for which there is a prime of good reduction $p_0$ with $a_E(p_0)\equiv i\pmod m$, then there exists $\nu_{E,m}\in\N$ such that if $0\leq\delta<\frac{1}{\nu_{E,m}}$ and $h\geq x^{1-\delta}$, then
\begin{equation}
\label{eqn:shiu-string}
\#\{p\in[x,x+h]:a_E(p)\equiv i~(\textup{mod}~m)\}\gg\frac{h}{\log x}.
\end{equation}
By using Theorem \ref{bounded-gaps-short} instead of \eqref{BO-short} in Balog and Ono's proof, we immediately obtain dense clusters of primes $p$ in short intervals for which $a_E(p)\equiv i\pmod m$.

\begin{corollary}
\label{shoe}
Let $E/\Q$ be an elliptic curve, let $m\in\N$, and let $i\bmod m$ be a residue class for which there is a prime of good reduction $p_0$ with $a_E(p_0)\equiv i\pmod m$.  There exist constants $\nu_{E,m}, C_{E,m}\in\N$ such that if $0\leq\delta<\frac{1}{\nu_{E,m}}$, $h\geq x^{1-\delta}$, $k\geq C_{E,m}$, and $\mathcal{H}_{k}$ is admissible, then
\begin{align*}
&\#\{n\in[x,x+h]:\#\{h_j\in\mathcal{H}_{k}:n+h_j\in\mathbb{P},a_E(n+h_j)\equiv i~(\bmod~m)\}\geq C_{E,m}^{-1}\log k\}\\
&\gg\frac{h}{(\log x)^k}.
\end{align*}
\end{corollary}

%%%%%%%%%%%%%%%%%%%%%%%%%%%%%%%%%%%%%%%%
\section*{Acknowledgments}
%%%%%%%%%%%%%%%%%%%%%%%%%%%%%%%%%%%%%%%%

The author thanks James Maynard, Robert Lemke Oliver, Ken Ono, Jeremy Rouse, Kannan Soundararajan, and the anonymous referees for their comments and suggestions.

%%%%%%%%%%%%%%%%%%%%%%%%%%%%%%%%%%%%%%%%
\section{Proof of Theorem \ref{main-theorem}}
\label{sec:proof_main_thm}
%%%%%%%%%%%%%%%%%%%%%%%%%%%%%%%%%%%%%%%%

For a number field $F$, we let $n_F=[F:\Q]$ and $d_F$ equal the absolute discriminant of $F$.  Let $L/K$ be a Galois extension of number fields with Galois group $G$, and let $C\subset G$ be a fixed conjugacy class.  Let $H$ be the largest abelian subgroup of $G$ such that $H\cap C$ is nonempty, and let $E$ be the field fixed by $H$.  If $L\cap\Q(\zeta_q)=\Q$, then $L( \zeta_q)/E$ is an abelian extension with Galois group $H_q$, which is isomorphic to $H\oplus (\Z/q\Z)^\times$.  Let $\chi$ be a Dirichlet character modulo $q$, and let $\xi$ be a Hecke character in the dual group $\widehat{H}$.  Since $L\cap\Q(\zeta_q)=\Q$, the characters $\omega$ in the dual group $\widehat{H}_q$ are of the form $\xi\otimes\chi$, and the conductor $\mathfrak{f}_{\omega}$ of $\omega$ satisfies $\mathrm{N}\mathfrak{f}_{\omega}\leq q^{n_K}\mathrm{N}\mathfrak{f}_{\xi}$, where $\mathrm{N}$ is the absolute field norm of $E$ (cf. \cite[Sections 0 and 1]{MP}).  Unless otherwise specified, all implied constants in the asymptotic notation $\ll$ or $O(\cdot)$ will depend in an effectively computable way on at most $\max_{\xi\in\widehat{H}}\mathrm{N}\mathfrak{f}_{\xi}$.

By Equation 3.2 of \cite{BO} and the functional equation for Hecke $L$-functions, if $y\leq h$, $\frac{1}{2}x\leq N\leq x$, and $T\leq x$, then
\begin{align*}
\max_{y\leq h}\max_{\frac{1}{2}x\leq N\leq x}\bigg|\psi_C(N+y;q,a)&-\psi_C(N;q,a)-\frac{|C|}{|G|}\frac{y}{\varphi(q)}\bigg|\\
&\ll \frac{h}{\varphi(q)}\sum_{\omega\in\widehat{H}_q}\sum_{\substack{\rho=\beta+i\gamma \\ L(\rho,\tilde{\omega})=0 \\ |\gamma|\leq T \\ \frac{1}{2}\leq \beta<1}}x^{\beta-1}+\frac{x(\log x)^2}{T},
\end{align*}
where $\tilde{\omega}$ is the primitive character which induces $\omega$.  Thus Theorem \ref{main-theorem} will follow from proving that for any fixed $D>0$, we have that
\begin{equation}
\label{eq1}
h\sideset{}{'}\sum_{q\leq Q}\frac{1}{\varphi(q)}\sideset{}{^*}\sum_{\omega\in\widehat{H}_q}\sum_{\substack{\rho_{\omega}=\beta_{\omega}+i\gamma_{\omega} \\ \frac{1}{2}\leq \beta_{\omega}<1 \\ |\gamma_{\omega}|\leq T}}x^{\beta_{\omega}-1}+\frac{Qx(\log x)^2}{T}\ll \frac{h}{(\log x)^D},
\end{equation}
where $\rho_{\omega}$ is a nontrivial zero of $L(s,\omega)$ and $\sum^*$ denotes summing over primitive characters $\omega$.  (See also \cite[Section 1]{MP} for a similar reduction.)  We now decompose the interval $[1,Q]$ into dyadic intervals of the form $[2^n,2^{n+1})$, where $0\leq n\leq\lceil\log_2 Q\rceil$.  Since there are $O(\log Q)$ such intervals and $\varphi(q)^{-1}\ll q^{-1}\log\log q$ for $q\geq6$, the left side of \eqref{eq1} is
\begin{equation}
%\begin{align}
\label{eq2}
%&\ll \sum_{n=0}^{\lceil\log_2 Q\rceil}\sideset{}{'}\sum_{q\in[2^n,2^{n+1})}\frac{1}{\varphi(q)}\sideset{}{^*}\sum_{\omega\in\widehat{H}_q}\sum_{\substack{\rho_{\omega}=\beta_{\omega}+i\gamma_{\omega} \\ \frac{1}{2}\leq \beta_{\omega}<1 \\ |\gamma_{\omega}|\leq T}}x^{\beta_{\omega}-1}+\frac{Qx(\log x)^2}{T}\notag\\
%&\ll 
(\log Q)(\log\log Q)\max_{1\leq Q_1\leq Q}\frac{1}{Q_1}\sideset{}{'}\sum_{q\leq Q_1}~\sideset{}{^*}\sum_{\omega\in\widehat{H}_q}\sum_{\substack{\rho_{\omega}=\beta_{\omega}+i\gamma_{\omega} \\ \frac{1}{2}\leq \beta_{\omega}<1 \\ |\gamma_{\omega}|\leq T}}x^{\beta_{\omega}-1}+\frac{Qx(\log x)^2}{T}.
%\end{align}
\end{equation}
If $\omega$ is primitive, then $\f_{\omega}$ is also the modulus of $\omega$.  Since $\mathrm{N}\f_{\omega}\leq q^{n_E}\max_{\xi\in\widehat{H}}\mathrm{N}\mathfrak{f}_{\xi}$, \eqref{eq2} is
\begin{equation}
\label{eqn:char_sum}
\ll (\log Q)(\log\log Q)\max_{1\leq Q_1\leq Q}\frac{1}{Q_1}\sum_{\mathrm{N}\a\leq Q_1^{n_E}}~\sideset{}{^*}\sum_{\omega\bmod\a}~\sum_{\substack{\rho_{\omega}=\beta_{\omega}+i\gamma_{\omega} \\ \frac{1}{2}\leq \beta_{\omega}<1 \\ |\gamma_{\omega}|\leq T}}x^{\beta_{\omega}-1}+\frac{Qx(\log x)^2}{T}.
\end{equation}

For $\frac{1}{2}\leq\sigma\leq 1$, let $N_{\omega}(\sigma,T):=\#\{\rho=\beta+i\gamma:L(\rho,\omega)=0,\sigma\leq\beta,|\gamma|\leq T\}$ and
\[
N(\sigma,R,T):=\sum_{\mathrm{N}\a\leq R}~~\sideset{}{^*}\sum_{\omega\bmod\a}N_{\omega}(\sigma,T).
\]
Building on the seminal work of Montgomery \cite[Theorem 12.2]{Montgomery-topics}, Hinz \cite{Hinz} proved estimates for $N(\sigma,Q^{n_E},T)$ when $n_E\geq2$.  The following proposition is a direct corollary of their combined work.
\begin{proposition}
\label{prop:Hecke_ZDE}
If $T\geq2$, $R\geq1$, and $\frac{1}{2}\leq\sigma\leq 1$, then
\[
N(\sigma,R,T)\ll (R^2 T^{n_E})^{\frac{5}{2}(1-\sigma)}(\log QT)^{9n_E+10}.
\]
\end{proposition}
\begin{proof}[Proof of Theorem \ref{main-theorem}]
Let $D$, $\delta$, and $h$ be as in the statement of Theorem \ref{main-theorem}.  Let $0<\epsilon<1$, and choose $Q=x^{\frac{2(1-\epsilon)-5n_E\delta}{15n_E}}(\log x)^{-\frac{D+2}{3}}$ and $T=x^{\frac{2(1-\epsilon+5n_E\delta)}{15n_E}}(\log x)^{\frac{2(D+2)}{3}}.$  With $1\leq Q_1\leq Q$, we have
\begin{align}
\label{eq4}
\sum_{\mathrm{N}\a\leq Q_1^{n_E}}~~\sideset{}{^*}\sum_{\omega\bmod\a}~\sum_{\substack{\rho_{\omega}=\beta_{\omega}+i\gamma_{\omega} \\ \frac{1}{2}\leq \beta_{\omega}<1 \\ |\gamma_{\omega}|\leq T}}x^{\beta_{\omega}-1}&\ll \log x\max_{\frac{1}{2}\leq\sigma< 1}x^{\sigma-1}N(\sigma,Q_1^{n_E},T).
\end{align}
By the zero-free region for Hecke $L$-functions proven by Bartz \cite{Bartz} and the fact that we restrict $q$ so that $L\cap\Q(\zeta_q)=\Q$, there exists a constant $b_L>0$ such that if
\begin{equation}
\label{ZFR-vino}
1-\eta(Q_1,x)<\sigma\leq 1,\qquad \eta(Q_1,x):=\frac{b_L}{\max\{\log Q_1,(\log x)^{3/4}\}},
\end{equation}
then $N(\sigma,Q_1^{n_E},T)$ is either 0 or 1.  If $N(\sigma,Q_1^{n_E},T)=1$, then the zero $\beta_1$ which is counted is a Siegel zero associated to an exceptional modulus $q_1$ and an exceptional real quadratic character in $\widehat{H}_{q_1}$.  As in \cite[Section 2]{MP}, a field-uniform version of Siegel's theorem for Hecke $L$-functions implies that $x^{\beta_1-1}\ll(\log x)^{-D-3}$ with an ineffective implied constant. 

Since $(Q^{2}T)^{\frac{5}{2}n_E}=x^{1-\epsilon}$, it follows from Proposition \ref{prop:Hecke_ZDE} that
\begin{align*}
\log x\max_{\frac{1}{2}\leq\sigma\leq 1-\eta(Q_1,T)}x^{\sigma-1}N(\sigma,Q_1^{n_E},T)&\ll(\log x)^{9n_E+11}\max_{\frac{1}{2}\leq\sigma\leq 1-\eta(Q_1,T)}((Q^2 T)^{\frac{5}{2}n_E}/x)^{1-\sigma}\notag\\
&\ll (\log x)^{9n_E+11} x^{-\epsilon\eta(Q_1,x)}.
\end{align*}
By our definition of $\eta(Q_1,x)$,
\begin{equation}
\label{eqn:eta_bound}
x^{-\epsilon\eta(Q_1,x)}\ll\begin{cases}
(\log x)^{-(9n_E+14+D)}&\mbox{if $1\leq Q_1\leq \exp((\log x)^{3/4})$},\\
1&\mbox{if $\exp((\log x)^{3/4})<Q_1\leq Q$}.
\end{cases}
\end{equation}
We have now bounded \eqref{eq4}, and so \eqref{eqn:char_sum} is bounded by
\[
h(\log Q)(\log\log Q)(\log x)\max_{Q_1\leq Q}\frac{1}{Q_1}((\log x)^{-D-3}+(\log x)^{9n_E+11} x^{-\epsilon\eta(Q_1,x)})+\frac{Q x(\log x)^2}{T}.
\]
For our choice of $h$, $Q$, and $T$, this is bounded by $h(\log x)^{-D}$ using \eqref{eqn:eta_bound}.
\end{proof}

%The full statement of Hinz's results \cite{Hinz} and Montgomery's results \cite{Montgomery-topics} is stronger than what is given in Proposition \ref{prop:Hecke_ZDE}; our weaker statement is more convenient.  One can use the full strength of Hinz's work to improve the ranges of $\delta$ and $\theta$ in Theorem \ref{main-theorem} by following the arguments of Huxley and Iwaniec \cite{HI}.  However, the resulting improvements are not much better than what is given in Theorem \ref{main-theorem}, and the resulting dependence of $\theta$ on $\delta$ and $n_E$ is a bit more complicated.  Hinz's density estimate depends on $n_E$ in such a way that one cannot match the quality of the results in \cite{HI} unless $n_E=1$.  This deficiency stems from summing over Hecke characters whose modulus has norm up to $O(Q^{n_E})$, which seems unavoidable at this time.

\section{Proof of Theorem \ref{bounded-gaps-short}}
\label{sec:proofs2}

We will use Theorem \ref{main-theorem} to prove Theorem \ref{bounded-gaps-short}.  Given a set of integers $\mathfrak{A}$, a set of primes $\mathfrak{P}\subset\mathfrak{A}$, and a linear form $L(n)=n+h$, define
\begin{equation*}
\begin{aligned}[c]
\mathfrak{A}(x)&=\{n\in\mathfrak{A}:x< n\leq 2x\},\\
L(\mathfrak{A})&=\{L(n):n\in\mathfrak{A}\},\\
\mathfrak{P}_{L,\mathfrak{A}}(x,y)&=L(\mathfrak{A}(x))\cap\mathfrak{P},
\end{aligned}
\begin{aligned}[c]
\mathfrak{A}(x;q,a)&=\{n\in\mathfrak{A}(x):n\equiv a~(\mathrm{mod}~q)\},\\
\varphi_{L}(q)&=\varphi(h q)/\varphi(h),\\
\mathfrak{P}_{L,\mathfrak{A}}(x;q,a)&=L(\mathfrak{A}(x;q,a))\cap\mathfrak{P}.
\end{aligned}
\end{equation*}
We consider the 6-tuple $(\mathfrak{A},\mathcal{L}_k,\mathfrak{P},B,x,\theta)$, where $\mathcal{H}_k$ is admissible, $\mathcal{L}_k=\{L_i(n)=n+h_i:h_i\in\mathcal{H}_k\}$, $B\in\N$ is constant, $x$ is a large real number, and $0\leq\theta<1$.  We present a very general hypothesis that Maynard states in Section 2 of \cite{maynard2}.

\begin{hypothesis}
\label{hyp}
With the above notation, consider the 6-tuple $(\mathfrak{A},\mathcal{H}_k,\mathfrak{P},B,x,\theta)$.
\begin{enumerate}
\item We have
\[
\sum_{q\leq x^\theta}\max_a\left|\#\mathfrak{A}(x;q,a)-\frac{\#\mathfrak{A}(x)}{q}\right|\ll \frac{\#\mathfrak{A}(x)}{(\log x)^{100k^2}}.
\]
\item For any $L\in \mathcal{H}_k$, we have
\[
\sum_{\substack{q\leq x^{\theta} \\ (q,B)=1}}\max_{(L(a),q)=1}\left|\#\mathfrak{P}_{L,\mathfrak{A}}(x;q,a)-\frac{\#\mathfrak{P}_{L,\mathfrak{A}}(x)}{\varphi_L(q)}\right|\ll \frac{\#\mathfrak{P}_{L,\mathfrak{A}}(x)}{(\log x)^{100k^2}}.
\]
\item For any $q\leq x^\theta$, we have $\#\mathfrak{A}(x;q,a)\ll \#\mathfrak{A}(x)/q.$
\end{enumerate}
\end{hypothesis}
\noindent
For $(\mathfrak{A},\mathcal{H}_k,\mathfrak{P},B,x,\theta)$ satisfying Hypothesis \ref{hyp}, Maynard proves the following in \cite{maynard2}.

\begin{theorem}
\label{big-maynard-thm}
Let $\alpha>0$ and $0\leq\theta<1$.  There is a constant $C$ depending only on $\theta$ and $\alpha$ so that the following holds.  Let $(\mathfrak{A},\mathcal{H}_k,\mathfrak{P},B,x,\theta)$ satisfy Hypothesis \ref{hyp}.  Assume that $C\leq k\leq (\log x)^\alpha$ and $h_i\leq x^\alpha$ for all $1\leq i\leq k$.  If $\delta>(\log k)^{-1}$ is such that
\[
\frac{1}{k}\frac{\varphi(B)}{B}\sum_{L_i\in\mathcal{H}_k}\#\mathfrak{P}_{L_i,\mathfrak{A}}(x)\geq\delta\frac{\#\mathfrak{A}(x)}{\log x},
\]
then
\[
\#\{n\in\mathfrak{A}(x):\#(\mathcal{H}_k(n)\cap\mathfrak{P})\geq C^{-1}\delta\log k\}\gg\frac{\#\mathfrak{A}(x)}{(\log x)^k \exp(Ck)}.
\]
\end{theorem}

\begin{proof}[Proof of Theorem \ref{bounded-gaps-short}]
The proof is essentially the same as Theorems 3.4 and 3.5 in \cite{maynard2}.  Let $\delta$, $h$, and $\theta$ be as in Theorem \ref{main-theorem}.  Let $\mathfrak{A}=\N\cap[x,x+h]$, $B=d_L$, and $\mathfrak{P}=\mathcal{P}$.  Parts (i) and (iii) of Hypothesis \ref{hyp} are trivial to check for the 6-tuple $(\N\cap[x,x+h],\mathcal{H}_{k},\mathcal{P},d_L,x,\theta/2)$.  By Theorem \ref{main-theorem} and partial summation, all of Hypothesis \ref{hyp} holds when $D$ and $x$ are sufficiently large in terms of $k$ and $\theta$.  Given a suitable constant $C_L>0$ (computed as in \cite{maynard, JT}), we let $k\geq C_L$.  For our choice of $\mathfrak{A}$ and $\mathfrak{P}$, we have the inequality
\[
\frac{1}{k}\frac{\varphi(d_L)}{d_L}\sum_{L_i\in\mathcal{H}_k}\#\mathfrak{P}_{L_i,\mathfrak{A}}(x)\geq(1+o(1))\frac{\varphi(d_L)}{d_L}\frac{|C|}{|G|}\frac{\#\mathfrak{A}(x)}{\log x}
\]
for all sufficiently large $x$, where the implied constant in $1+o(1)$ depends only on $L$.  Theorem \ref{bounded-gaps-short} now follows directly from Theorem \ref{big-maynard-thm}.
\end{proof}

\bibliographystyle{abbrv}
\bibliography{JAThorner_BV_Short_Intervals}

\end{document}